\documentclass{article}

\usepackage{amsmath,amsthm,amssymb,amsfonts}
\usepackage{verbatim}
\usepackage{stmaryrd} 
\usepackage{hyperref}
\usepackage{tikz}
\usepackage{tkz-berge, tkz-graph}
\tikzset{vertex/.style={circle,draw,fill,inner sep=0pt,minimum size=1mm}}
\theoremstyle{plain}
\newtheorem{thm}{Theorem}
\newtheorem{lem}[thm]{Lemma}
\newtheorem{prop}[thm]{Proposition}

\newtheorem{cor}[thm]{Corollary}

\theoremstyle{definition}

\newtheorem{defn}[thm]{Definition}
\newtheorem{exl}[thm]{Example}

\newtheorem{quest}[thm]{Question}

\numberwithin{thm}{section}

\newcommand{\adj}{\leftrightarrow}
\newcommand{\adjeq}{\leftrightarroweq}

\DeclareMathOperator{\im}{im}
\DeclareMathOperator{\id}{id}
\newcommand{\pd}{\partial}
\DeclareMathOperator{\NP}{NP}
\def\R{{\mathbb R}}
\newcommand{\Z}{\mathbb{Z}}

\begin{document}

\title{Strong homotopy of digitally continuous functions}
\author{P. Christopher Staecker}

\maketitle

\begin{abstract}
We introduce a new type of homotopy relation for digitally continuous functions which we call ``strong homotopy.'' Both digital homotopy and strong homotopy are natural digitizations of classical topological homotopy: the difference between them is analogous to the difference between digital 4-adjacency and 8-adjacency in the plane.

We explore basic properties of strong homotopy, and give some equivalent characterizations. In particular we show that strong homotopy is related to ``punctuated homotopy,'' in which the function changes by only one point in each homotopy time step.

We also show that strongly homotopic maps always have the same induced homomorphisms in the digital homology theory. This is not generally true for digitally homotopic maps, though we do show that it is true for any homotopic selfmaps on the digital cycle $C_n$ with $n\ge 4$.

We also define and consider strong homotopy equivalence of digital images. Using some computer assistance, we produce a catalog of all small digital images up to strong homotopy equivalence. We also briefly consider pointed strong homotopy equivalence, and give an example of a pointed contractible image which is not pointed strongly contractible.
\end{abstract}

\section{Introduction}
A \emph{digital image} is a set of points $X$, with some \emph{adjacency relation} $\kappa$, which is symmetric and antireflexive. The standard notation for such a digital image is $(X, \kappa)$. Typically in digital topology the set $X$ is a finite subset of $\Z^n$, and the adjacency relation is based on some notion of adjacency of points in the integer lattice. 

We will use the notation $x \adj_\kappa y$ when $x$ is adjacent to $y$ by the adjacency $\kappa$, and $x \adjeq_\kappa y$ when $x$ is adjacent or equal to $y$. The particular adjacency relation will usually be clear from context, and in this case we will omit the subscript.

The results of this paper hold generally, without any specific reference to the embedding of the images in $\Z^n$. Thus we will often simply consider a digital image $X$ as an abstract simple graph, where the vertices are points of $X$, and an edge connects two vertices $x, x' \in X$ whenever $x\adj x'$. 

\begin{defn} 
Let $(X, \kappa), (Y, \lambda)$ be digital images. A function $f: X \rightarrow Y$ is $(\kappa, \lambda)$-continuous iff whenever $x \adj_{\kappa} y$ then $f(x) \adjeq_{\lambda} f(y)$. 
\end{defn}

For simplicity of notation, we will generally not need to reference the adjacency relation specifically. Thus we typically will denote a digital image simply by $X$, and when the appropriate adjacency relations are clear we simply call a function between digital images ``continuous''. We will often refer to digital images as simply ``images''.

For any digital image $X$ the identity map $\id_X:X \to X$ is always continuous.

The topological theory of digital images has, to a large part, been characterized by taking ideas from classical topology and ``discretizing'' them. Typically $\R^n$ is replaced by $\Z^n$, and so on. 
Viewing $\Z$ as a digital image, it makes sense to use the following adjacency relation: $a,b\in \Z$ are called \emph{$c_1$-adjacent} if and only if $|a-b|=1$. This adjacency relation corresponds to connectivity in the standard topology of $\R$. (This adjacency is called $c_1$ because it is the first in a class of standard adjacencies on $\Z^n$, the notation first appears in \cite{han05}.)

The following is the standard definition of digital homotopy. For $a,b\in \Z$ with $a\le b$, let $[a,b]_\Z$ be the \emph{digital interval} $[a,b]_\Z = \{a, a+1, \dots, b\}$.

\begin{defn}\cite{boxe99}
\label{homotopydef}
Let $(X,\kappa)$ and $(Y,\lambda)$ be digital images.
Let $f,g: X \rightarrow Y$ be $(\kappa,\lambda)$-continuous functions.
Suppose there is a positive integer $m$ and a function
$H: X \times [0,k]_{\Z} \rightarrow Y$
such that:

\begin{itemize}
\item for all $x \in X$, $H(x,0) = f(x)$ and $H(x,k) = g(x)$;
\item for all $x \in X$, the induced function
      $H_x: [0,k]_{\Z} \rightarrow Y$ defined by
          \[ H_x(t) ~=~ H(x,t) \mbox{ for all } t \in [0,k]_{\Z} \]
          is $(c_1,\lambda)-$continuous. 
\item for all $t \in [0,k]_{\Z}$, the induced function
         $H_t: X \rightarrow Y$ defined by
          \[ H_t(x) ~=~ H(x,t) \mbox{ for all } x \in  X \]
          is $(\kappa,\lambda)-$continuous.
\end{itemize}
Then $H$ is a {\em [digital] homotopy between} $f$ and
$g$, and $f$ and $g$ are {\em [digitally] homotopic}, denoted
$f \simeq g$.

If $k=1$, then $f$ and $g$ are {\em homotopic in one step}.
\end{defn}

Homotopy in one step can easily be expressed in terms of individual adjacencies:
\begin{prop}\label{htpadjs}
Let $f,g:X\to Y$ be continuous. Then $f$ is homotopic to $g$ in one step if and only if for every $x\in X$, we have $f(x)\adjeq g(x)$.
\end{prop}
\begin{proof}
Assume that $f$ is homotopic to $g$ in one step. The homotopy is simply defined by $H(x,0)=f(x)$ and $H(x,1)=g(x)$. Then by continuity in the second coordinate of $H$, we have $f(x) = H(x,0) \adjeq H(x,1) = g(x)$ as desired.

Now assume that $f(x) \adjeq g(x)$ for each $x$. Define $H:X\times [0,1]_\Z \to Y$ by $H(x,0)=f(x)$ and $H(x,1)=g(x)$. Clearly $H(x,t)$ is continuous in $x$ for each fixed $t$, since $f$ and $g$ are continuous. Also $H(x,t)$ is continuous in $t$ for fixed $x$ because $H(x,0) = f(x) \adjeq g(x) = H(x,1)$. Thus $H$ is a one step homotopy from $f$ to $g$ as desired.
\end{proof}

Definition \ref{homotopydef} is inspired by the concept of homotopy from classical topology, 
but the classical definition is simpler because it can use the product topology. 
In classical topology the continuity of the two types of ``induced function'' is expressed simply by saying that $H:X\times [0,1] \to Y$ is continuous with respect to the product topology on $X \times [0,1]$.

Given two digital images $A$ and $B$, we can consider the product $A\times B$ as a digital image, but there are several choices for the adjacency to be used. The most natural adjacencies are the \emph{normal product adjacencies}, which were defined by Boxer in \cite{boxe17}. 
\begin{defn}\cite{boxe17}
For each $i\in \{1,\dots,n\}$, let $(X_i,\kappa_i)$ be digital images. Then for some $u \in \{1,\dots,n\}$, the \emph{normal product adjacency} $\NP_u(\kappa_1,\dots,\kappa_n)$ is the adjacency relation on $\prod_{i=1}^n X_i$ defined by: $(x_1,\dots, x_n)$ and $(x'_1,\dots, x'_n)$ are adjacent if and only if their coordinates are adjacent in at most $u$ positions, and equal in all other positions.

When the underlying adjacencies are clear, we abbreviate $\NP_u(\kappa_1,\dots,\kappa_n)$ as simply $\NP_u$.
\end{defn}

The normal product adjacency is inspired by the various standard adjacencies typically used on $\Z^n$. On $\Z^1$, as mentioned above, the standard adjacency is $c_1$. On $\Z^2$, the typical adjacencies are \emph{4-adjacency}, in which each point is adjacent to its 4 horizontal and vertical neighbors, and \emph{8-adjacency}, in which each point is additionally adjacent to its diagonal neighbors. Viewing $\Z^2$ as the product $\Z^2 = \Z \times \Z$, it is easy to see that 4-adjacency is the same as $\NP_1(c_1,c_1)$, and 8-adjacency is $\NP_2(c_1,c_1)$. The typical adjacencies used in $\Z^3$ are 6-, 18-, and 26-adjacency, depending on which types of diagonal adjacencies are allowed. These adjacencies are exactly $\NP_1(c_1,c_1,c_1)$, $\NP_2(c_1,c_1,c_1)$, and $\NP_3(c_1,c_1,c_1)$. 

Boxer showed that the definition of homotopy can be rephrased in terms of an $\NP_1$ product adjacency:
\begin{thm}\cite[Theorem 3.6]{boxe17}\label{np1thm}
Let $(X,\kappa)$ and $(Y,\lambda)$ be digital images. Then $H:X\times [0,k]_\Z \to Y$ is a homotopy if and only if $H$ is $(\NP_1(\kappa,c_1),\lambda)$-continuous.
\end{thm}

In this paper we will propose a new type of homotopy relation and explore its properties. In Section \ref{strongsection} we define strong homotopy, which is our main object of study. In Section \ref{punctuatedsec} we define punctuated homotopy, in which the function changes by only one point for each time step, and we show that functions are strong homotopic if and only if they are punctuated homotopic. In Section \ref{homologysec} we discuss the digital homology groups, and show that they are well-behaved with respect to strong homotopy and strong homotopy equivalence. In Section \ref{catalogsec} we produce a catalog of small digital images up to strong homotopy equivalence, inspired by the work of \cite{hmps15}. In Section \ref{pointedsec} we consider pointed strong homotopy equivalence, and show that it is not always equivalent to nonpointed strong homotopy equivalence.

\section{Strong homotopy}\label{strongsection}

Theorem \ref{np1thm} states that a homotopy is a function $H:X\times [0,k]_\Z \to Y$ which is continuous when we use the $\NP_1$ product adjacency in the domain. In this paper we explore the following definition, which simply uses $\NP_2$ in place of $\NP_1$. As we will see this imposes extra restrictions on the homotopy $H$.

\begin{defn}
Let $(X,\kappa)$ and $(Y,\lambda)$ be digital images. We say $H:X\times [0,k]_\Z \to Y$ is a \emph{strong homotopy} when $H$ is $(\NP_2(\kappa,c_1), \lambda)$-continuous.

If there is a strong homotopy $H$ between $f$ and $g$, we say $f$ and $g$ are \emph{strongly homotopic}, and we write $f \cong g$. If additionally $k=1$, we say $f$ and $g$ are \emph{strongly homotopic in one step}.
\end{defn}

Since we have exchanged $NP_1$ for $NP_2$ in the definition above, we may say informally that strong homotopy and digital homotopy provide two different but equally natural ``digitizations'' of the classical topological idea of homotopy. The difference between homotopy and strong homotopy of continuous functions is analogous to the difference between 4-adjacency and 8-adjacency of points in the plane.

It is clear from the definition of the normal product adjacency that if two points are $\NP_1$-adjacent, then they are $\NP_2$-adjacent. Thus any function $f:A\times B \to C$ which is continuous when using $\NP_2$ in the domain will automatically be continuous when using $\NP_1$ in the domain. Thus we obtain the following, which justifies the use of the word ``strong.''
\begin{thm}\label{strongimplieshtp}
Let $H:X\times [0,k]_\Z \to Y$ be a strong homotopy. Then $H$ is a homotopy.
\end{thm}

A standard argument shows that strong homotopy is an equivalence relation.
\begin{thm}
Strong homotopy is an equivalence relation.
\end{thm}
\begin{proof}
For reflexivity, given any continuous map $f:(X,\kappa)\to (Y,\lambda)$, the function $H:X\times [0,1]_\Z \to Y$ given by $H(x,t) =f(x)$ is a strong homotopy from $f$ to itself.

For symmetry, let $f\cong g$ by some strong homotopy $H:X\times [0,k]_\Z \to Y$ with $H(x,0)=f(x)$ and $H(x,k) =g(x)$. Let $r:[0,k]_\Z \to [0,k]_\Z$ be $r(t) = k-t$. Then $r$ is a $c_1$-isomorphism, and thus $\id_X \times r:X \times [0,k]_\Z \to X \times [0,k]_\Z$ is an $\NP_2(\kappa,c_1)$-isomorphism. (Theorem 4.1 of \cite{boxe17} shows that any normal product of functions is an isomorphism if and only if each factor map is an isomorphism.) Then $H \circ (\id_X \times r): X\times [0,k]_\Z \to Y$ is $(\NP_2(\kappa,c_1),\lambda)$-continuous, and thus a strong homotopy from $g$ to $f$.

For transitivity, let $f \cong g$ by some $H:X\times [0,k]_\Z \to Y$ and $g \cong h$ by $G:X\times [0,l]_\Z \to Y$. Then let $\bar H:X\times [0,k+l+1]_\Z \to Y$ be given by:
\[ \bar H(x,t) = \begin{cases} H(x,t) \quad &\text{ if } t \le k, \\
G(x,t-k-1) &\text{ if } t>k. \end{cases} \]
We need only show that $\bar H$ is $(\NP_2(\kappa,c_1),\lambda)$-continuous, and then it will be a strong homotopy from $f$ to $h$.

Take $(x,t) \adj_{\NP_2} (x',t')$, and we will show that $\bar H(x,t) \adjeq_\lambda \bar H(x',t')$. If $t$ and $t'$ are both less than or equal to $k$, then $\bar H(x,t) = H(x,t)$ and $\bar H(x',t')=H(x',t')$, and so $\bar H(x,t) \adjeq_\lambda \bar H(x',t')$ because $H$ is $(\NP_2(\kappa,c_1),\lambda)$-continuous. Similarly $\bar H(x,t) \adjeq_\lambda \bar H(x',t')$ when both $t$ and $t'$ are greater than $k$.

Without loss of generality, it remains only to consider when $t\le k$ and $t' > k$. Since $(x,t) \adj_{\NP_2} (x',t')$ we must have $t\adjeq_{c_1} t'$ and thus we must have $t=k$ and $t'=k+1$. Then we have $\bar H(x,t) = \bar H(x,k) = H(x,k) = g(x)$ and $\bar H(x',t') = \bar H(x',k+1) = G(x',0) = g(x')$, and since $x\adjeq x'$ and $g$ is continuous we have $\bar H(x,t) \adjeq_\lambda \bar H(x',t')$ as desired.
\end{proof}

Strong homotopy in one step can be expressed in terms of adjacencies as in Proposition \ref{htpadjs}.
\begin{thm}\label{stronghtpadjs}
Let $f,g:X\to Y$ be continuous. Then $f$ is strongly homotopic to $g$ in one step if and only if for every $x,x'\in X$ with $x\adj x'$, we have $f(x)\adjeq g(x')$.
\end{thm}
\begin{proof}
First assume that $f$ is homotopic to $g$ in one step. The homotopy is simply defined by $H(x,0)=f(x)$ and $H(x,1)=g(x)$. Then if $x \adj x'$, we will have $(x,0) \adj_{\NP_2} (x',1)$, and thus since $H$ is $\NP_2$-continuous we have $f(x) = H(x,0) \adjeq H(x',1) = g(x')$ as desired.

Now assume that $f(x) \adjeq g(x')$ for each $x$. Define $H:X\times [0,1]_\Z \to Y$ by $H(x,0)=f(x)$ and $H(x,1)=g(x)$, and we must show that $H$ is $\NP_2$-continuous. Take $(x,t),(x',t') \in X\times [0,1]_\Z$ with $(x,t)\adj_{\NP_2} (x',t')$, and we will show that $H(x,t)\adjeq H(x',t')$. We have a few cases based on the values of $t,t'\in \{0,1\}$.

If $t=t'$, without loss of generality say $t=t'=0$. Since $(x,t)\adj_{\NP_2} (x',t')$, we have $x\adjeq x'$, and so we have 
\[ H(x,t) = H(x,0) = f(x) \adjeq f(x') = H(x',0) = H(x',t') \]
since $f$ is continuous. Thus $H(x,t)\adjeq H(x',t')$ as desired.

Finally, if $t\neq t'$, without loss of generality assume $t=0$ and $t'=1$. Since $(x,t)\adj_{\NP_2} (x',t')$, we have $x\adjeq x'$, and so we have 
\[ H(x,t) = H(x,0) = f(x) \adjeq g(x') = H(x',1) = H(x',t') \]
and so $H(x,t)\adjeq H(x',t')$ as desired.
\end{proof}

By Theorem \ref{strongimplieshtp}, if $f \cong g$ then automatically we have $f \simeq g$. The converse is not true, however, as the following example shows.

One important source of examples in the study of digital images is the \emph{digital cycle} $C_n = \{c_0,\dots, c_{n-1}\}$, with adjacency given by $c_i \adj c_{i+1}$ for each $i$, where for convenience we always read the subscripts modulo $n$. Thus $C_n$ is a digital image of $n$ points which is in many ways topologically analogous to the circle.

\begin{exl}\label{strongrigid}
It is well known that all selfmaps on $C_4$ are homotopic to one another. But we will show that the identity map $\id_{C_4}:C_4 \to C_4$ is not strongly homotopic to any map $f$ with $\#f(C_4) < 4$. It will suffice to show that $\id_{C_4}$ is not strongly homotopic in 1 step to any such map. 

Without loss of generality assume that $f(C_4) \subseteq \{c_0,c_1,c_2\}$, and assume for the sake of a contradiction that $f$ is strongly homotopic in one step to $\id_{C_4}$. Then by Theorem \ref{stronghtpadjs}, since $c_2 \adj c_3$ and $c_0\adj c_3$, we will have $c_2 \adjeq f(c_3)$ and also $c_0 \adjeq f(c_3)$. Thus $f(c_3)$ is adjacent to both $c_0$ and $c_2$, but cannot equal $c_3$ since $c_3 \not\in f(C_4)$. We conclude that $f(c_3)=c_1$. By Theorem \ref{htpadjs}, this contradicts the fact that $f$ is homotopic to the identity in 1 step.
\end{exl}

\section{Punctuated homotopy}\label{punctuatedsec}
Given a (not necessarily strong) homotopy $H(x,t)$, we say that $H$ is \emph{punctuated} if, for each $t$, there is some $x_t$ such that $H(x,t) = H(x,t+1)$ for all $x\neq x_t$. That is, from each stage of the homotopy to the next, the induced map $H_t(x)$ is changing by one point at a time. 

\begin{thm}
Any punctuated homotopy is a strong homotopy.
\end{thm}
\begin{proof}
Let $H$ be a punctuated homotopy, and let $(x,t) \adj_{\NP_2} (x',t')$. We must show that $H(x,t) \adjeq H(x',t')$. Since $(x,t) \adj_{\NP_2} (x',t')$ we have $x \adjeq x'$ and $t \adjeq t'$. If $t=t'$, then we have $(x,t) \adj_{\NP_1} (x',t')$ and so $H(x,t) \adjeq H(x',t')$ since $H$ is a homotopy. It remains to consider when $t \adj t'$ and $t\neq t'$. In this case, without loss of generality assume $t'=t+1$.

Since $H$ is a punctuated homotopy, there is some $x_t$ such that $H(x,t) = H(x,t+1)$ for all $x\neq x_t$. 
When $x \neq x_t$, we have 
\[ H(x,t) = H(x,t+1) \adjeq H(x',t+1) = H(x',t') \]
since $H$ is a homotopy. Similarly we have $H(x,t) \adjeq H(x',t')$ when $x' \neq x_t$.

Thus it remains only to consider when $x=x'=x_t$. That is, we must show that $H(x_t,t) \adjeq H(x_t,t') = H(x_t,t+1)$, and this is true because $H$ is a homotopy.
\end{proof}

We also have a sort of converse to the above. While not every strong homotopy is punctuated, any two strongly homotopic maps can be connected by a punctuated homotopy, provided that the domain is finite.

\begin{thm}\label{punctuatedthm}
Let $X$ be a finite digital image, and let $f,g:X\to Y$ be strongly homotopic. Then $f$ and $g$ are homotopic by a punctuated homotopy.
\end{thm}
\begin{proof}
By induction, it suffices to show that if $f$ and $g$ are strongly homotopic in one step, then they are homotopic by a punctuated homotopy. Enumerate the points of $X$ as $X = \{x_0,\dots, x_n\}$, and define $H:X\times [0,n]_\Z \to Y$ by:
\[ H(x_i,t) = \begin{cases} f(x_i)  & \quad\text{ if } i \ge t, \\
g(x_i) &\quad\text{ if } i < t. \end{cases} \]
Then $H$ moves one point at a time, so we need only to show that it has the appropriate continuity properties to be a homotopy. 

First we show that $H(x,t)$ is continuous in $x$ for fixed $t$. Take $x \adj x'$, then $H(x,t) \in \{f(x),g(x)\}$ and $H(x',t) \in \{f(x'),g(x')\}$. Since $f$ and $g$ are homotopic in one step we have $f(x)\adjeq f(x')$ and $g(x) \adjeq g(x')$, and since $f$ and $g$ are strongly homotopic in one step we have $f(x) \adjeq g(x')$ and $g(x) \adjeq f(x')$. Thus in any case we have $H(x,t) \adjeq H(x',t)$ as desired.

Now we show that $H(x,t)$ is continuous in $t$ for fixed $x$. It suffices to show that $H(x,t) \adjeq H(x,t+1)$ for any $t$. We have $H(x,t) \in \{f(x),g(x)\}$ and also $H(x,t+1) \in \{f(x),g(x)\}$. Since $f$ and $g$ are homotopic in one step we have $f(x)\adjeq g(x)$, and thus $H(x,t) \adjeq H(x,t+1)$ as desired.
\end{proof}

The finiteness assumption above is necessary, as the following example shows.
\begin{exl}
Let $X = \Z \times \{0,1\} \subset \Z^2$, with 8-adjacency, and let $f(x,y) = (x,0)$. Then $f$ is strongly homotopic to $\id_X$ in one step. But $f(x,y)$ and $\id(x,y)$ differ for infinitely many values of $(x,y)\in X$. Since a puncuated homotopy has finite time interval, and can only change one value at a time, there can be no punctuated homotopy from $f$ to $\id_X$. 
\end{exl}

Combining the two theorems above gives us a powerful characterization of strong homotopy. 
\begin{cor}
If $X$ is finite, two continuous maps $f,g:X\to Y$ are strongly homotopic if and only if they are homotopic by a punctuated homotopy.
\end{cor}

As an application, we show that the identity map on the $n$-cycle for $n\ge 4$ is not strongly homotopic to any other map:
\begin{exl}
Let $n \ge 4$, and assume for the sake of a contradiction that there is some continuous $f:C_n \to C_n$ with $f \cong \id_{C_n}$ and $f\neq \id_{C_n}$. Without loss of generality, assume that the homotopy from $f$ to $\id_{C_n}$ is punctuated and in one step. Since the homotopy is punctuated, $f$ moves one point, so without loss of generality we may assume that $f(c_0)=c_1$ and $f(c_i)=c_i$ for all $i\neq 0$. This contradicts the continuity of $f$, however, since we will have $c_0\adj c_{n-1}$ but $c_1 = f(c_0) \not \adjeq f(c_{n-1}) = c_{n-1}$ since $n\ge 4$. 
\end{exl}

%
%

\section{Homology groups and strong homotopy}\label{homologysec}
We begin this section with a brief review of the digital homology theory, as presented in \cite{bko11}. For a digital image $X$ and some $q\ge 1$, a $q$-simplex is defined to be any set of $q+1$ mutually adjacent points of $X$. For mutually adjacent points $x_0,\dots,x_q$, the associated $q$-simplex is denoted $\langle x_0,\dots, x_q\rangle$. 

The \emph{chain group} $C_q(X)$ is defined to be the free abelian group generated by the set of all $q$-simplices. If $\rho:\{0,\dots,q\} \to \{0,\dots,q\}$ is a permutation, then in $C_q(X)$ we identify
\[ \langle x_0,\dots, x_q \rangle = (-1)^\rho \langle x_{\rho(0)},\dots,x_{\rho(q)}\rangle, \]
where $(-1)^\rho = 1$ when $\rho$ is an even permutation, and $(-1)^\rho=-1$ when $\rho$ is an odd permutation.

The \emph{boundary homomorphism} $\pd_q: C_q(X) \to C_{q-1}(X)$ is the homomorphism induced by defining:
\[ \pd_q(\langle x_0,\dots,x_q\rangle) = \sum_{i=0}^q (-1)^i \langle x_0,\dots,\widehat{x_i},\dots x_q\rangle, \]
where $\widehat{x_i}$ indicates omission of the $x_i$ coordinate.

This $\pd_q$ gives an exact sequence:
\[ \dots \to C_{q+1}(X) \xrightarrow{\pd_{q+1}} C_q(X) \xrightarrow{\pd_q} C_{q-1}(X) \to \dots C_0(X) \to 0 \]
For each $q$, the group of \emph{$q$-cycles} is defined to be $Z_q(X) = \ker \pd_q \subset C_q(X)$, and the group of \emph{$q$-boundaries} is the group $B_q(X) = \im \pd_{q+1} \subset C_q(X)$. Since the sequence above is exact, $Z_q(X)$ is a subgroup of $B_q(X)$, and the \emph{dimension $q$ homology group} is defined as $H_q(X) = B_q(X) / Z_q(X)$.

Any continuous function $f:X\to Y$ induces a homomorphism $f_{\#,q}:C_q(X) \to C_q(Y)$ defined by
\[ f_{\#,q}(\langle x_0,\dots, x_q\rangle) = \langle f(x_0),\dots,f(x_q)\rangle, \]
where the right side is interpreted as 0 if the set $\{f(x_0),\dots,f(x_q)\}$ has cardinality less than $q$. Usually the value of $q$ is understood, and we simply write $f_{\#,q} = f_\#$. 

Theorem 3.15 of \cite{bko11} shows that such a map $f$ induces a homomorphism $f_{*,q}:H_q(X) \to H_q(X)$ defined by $f_{*,q}(\sigma + B_q(X)) = f_{\#,q}(\sigma) + B_q(X)$ for $\sigma \in C_q(X)$, and that this assignment is functorial, in the sense that $(f\circ g)_* = f_* \circ g_*$ for all $q$. Again, we typically write $f_{*,q} = f_*$ when the $q$ is understood.

These definitions and results are all exactly as expected from the classical homology theory of a simplicial complex. As an example, we compute the homology groups of the cycle $C_n$ for $n\ge 4$.
\begin{thm}\label{Cnhomology}
If $n\ge 4$, we have:
\[ H_q(C_n) = \begin{cases} \Z \quad &\text{ if $q\le 1$,} \\
0 &\text{ if $q>1$.} \end{cases} \]
\end{thm}
\begin{proof}
First we prove the case $q=0$. The chain group $C_0(C_n)$ is generated by $n$ different $0$-simplices $\langle c_0\rangle, \dots, \langle c_{n-1}\rangle$. Since $\pd_0$ is a trivial homomorphism, we have $Z_0(C_n) = C_0(C_n)$. Note that for each $i$, we have 
\[ \langle c_i\rangle = (\langle c_i \rangle - \langle c_{i+1} \rangle) + \langle c_{i+1} \rangle = \pd \langle c_{i+1}, c_{i}\rangle + \langle c_{i+1} \rangle, \]
and thus $\langle c_i\rangle - \langle c_{i+1}\rangle \in B_0(C_n)$. Thus $\langle c_i \rangle$ and $\langle c_{i+1}\rangle$ are equal in $H_0(C_n)$ for every $i$. That is, $H_0(C_n)$ is the group generated by $\langle c_0\rangle$, and so $H_0(C_n) = \Z$ as desired.

Now for $q=1$, first we note that there are no $2$-simplices in $C_n$ (because $n\ge 4$), so $C_2(C_n)$ is trivial, and thus $B_1(C_n)$ is trivial. Thus $H_1(C_n)$ will be isomorphic to $Z_1(C_n)$. To determine $Z_1(C_n)$, we must determine which $\alpha \in C_1(C_n)$ satisfy $\pd \alpha = 0$. Any $\alpha\in C_1(C_n)$ can be expressed as:
\[ \alpha = w_1\langle c_0,c_1\rangle + \dots + w_n\langle c_{n-1}, c_0\rangle \]
for $w_i \in \Z$, and then $\pd \alpha = 0$ if and only if:
\begin{align*}
0 &= \pd \alpha = w_1(\langle c_1\rangle - \langle c_0\rangle) + \dots + w_n(\langle c_{0}\rangle - \langle c_{n-1}\rangle) \\
&= (w_n-w_1) \langle c_0 \rangle + (w_1-w_2)\langle c_1 \rangle + \dots + (w_{n-1}-w_n)\langle c_{n-1} \rangle
\end{align*}
and thus we have $w_1=w_2=\dots=w_n$ since the $\langle c_i\rangle$ are linearly independent in $C_1(C_n)$. Then we have shown that $H_1(C_n) = Z_1(C_n)$ is generated by the single element 
\[ \sigma = \sum_{i=0}^{n-1} \langle c_i,c_{i+1}\rangle, \]
and thus $H_1(C_n) = \Z$. 

For $q>1$, there are no $q$-simplices and so $C_q(C_n)$ is trivial, and thus $H_q(C_n)$ is trivial.
\end{proof}

One major difference between the digital theory and classical homology is that the induced homomorphism $f_*$ is not always a digital homotopy invariant.

\begin{exl}
By Theorem \ref{Cnhomology}, the homology group $H_1(C_4)$ is isomorphic to $\Z$.
Because all maps on $C_4$ are homotopic, the identity map is homotopic to a constant map $c$. But it is easy to see that $\id_*:H_1(C_4) \to H_1(C_4)$ is the identity homomorphism of $\Z$, while $c_*:H_1(C_4) \to H_1(C_4)$ is the trivial homomorphism. Thus $\id \simeq c$ but $\id_* \neq c_*$.
\end{exl}

The lack of a homotopy-invariant induced homorphism is a major deficiency in the homology theory of digital images. Lacking this homotopy invariance, the homology groups are not well-behaved with respect to typical topological constructions. For example two homotopy equivalent digital images may have different homology groups.

When we restrict our attention to strong homotopy, however, the induced homomorphism is invariant. The proof below requires $X$ to be finite because we wish to use Theorem \ref{punctuatedthm}. We believe that the theorem holds even for infinite domains, but we will not try to prove it.

\begin{thm}\label{prism}
If $X$ is finite and $f,g:X\to Y$ are strongly homotopic, then the induced homomorphisms $f_*,g_*:H_q(X) \to H_q(Y)$ are equal for each $q$.
\end{thm}
\begin{proof}
By induction and Theorem \ref{punctuatedthm}, it suffices to prove the result when $f$ and $g$ are homotopic by a punctuated homotopy in one step. 
We mimic the proof for this result in classical homology theory, see for example Proposition 2.10 of \cite{hatc02}. We define the ``prism operator'' $P:C_q(X) \to C_{q+1}(Y)$ as follows: For $\sigma\in C_q(X)$ with $\sigma = \langle x_0,\dots,x_q\rangle$, let:
\[ P(\sigma) = \sum_{j=0}^q (-1)^j \langle f(x_0),\dots, f(x_j), g(x_j),\dots, g(x_q)\rangle, \]
where the term $\langle f(x_0),\dots, f(x_j), g(x_j),\dots, g(x_q)\rangle$ is interpreted as $0$ if any of these points are equal.

Note that the definition of $P$ only makes sense if $\{f(x_0),\dots, f(x_j), g(x_j),\dots, g(x_q)\}$ is indeed a set of $q+1$ points which are mutually adjacent or equal. This is ensured because $f$ and $g$ are strongly homotopic in 1 step, and thus by Theorem \ref{stronghtpadjs}, since $\{x_0,\dots, x_q\}$ are mutually adjacent, the points of $\{f(x_0),\dots, f(x_q),g(x_0),\dots, g(x_q)\}$ are mutually adjacent or equal. This is the point at which the proof will fail if the homotopy is not strong.

The bulk of the proof consists of proving the following formula:
\begin{equation}\label{prismeq} 
\pd (P(\sigma)) = g_\#(\sigma) - f_\#(\sigma) - P(\pd \sigma). 
\end{equation}

Since $f$ and $g$ are homotopic in one step by punctuated homotopy, there is some $x'\in X$ such that $f(x)=g(x)$ for all $x\neq x'$. Formula \eqref{prismeq} is easy to prove when $\sigma$ does not use the vertex $x'$:

If $\sigma = \langle x_0,\dots,x_q\rangle$ with $x_i \neq x'$ for all $i$, then $f(x_i)=g(x_i)$ for each $i$. Thus $P(\sigma) = 0$, since $\langle f(x_0),\dots, f(x_j), g(x_j),\dots, g(x_q)\rangle$ will repeat the point $f(x_j)=g(x_j)$. Thus the left side of \eqref{prismeq} is 0. For the right side, note that $\pd \sigma$ also does not use the point $x'$, and so for the same reasons we will have $P(\pd \sigma) = 0$. Also since $\sigma$ does not use $x'$, we will also have $f_\#(\sigma) = g_\#(\sigma)$, and thus the right side of \eqref{prismeq} is also 0, and we have proved \eqref{prismeq}.

Now we prove \eqref{prismeq} in the case when $\sigma$ does use the point $x'$. Without loss of generality assume $\sigma = \langle x', x_1,\dots,x_q\rangle$. In this case we have
\begin{align*} 
P(\sigma) &= \langle f(x'),g(x'),g(x_1),\dots, g(x_q) \rangle \\
&\qquad + \sum_{j=1}^q (-1)^j \langle f(x'),f(x_1),\dots, f(x_j), g(x_j), \dots, g(x_q) \rangle \\
&= \langle f(x'),g(x'),g(x_1),\dots, g(x_q) \rangle
\end{align*}
where most of the terms above are 0 because they repeat the point $f(x_j)=g(x_j)$. 
Then the left side of \eqref{prismeq} is
\begin{align*}
\pd (P(\sigma)) &= \pd(\langle f(x'),g(x'), g(x_1), \dots, g(x_q)\rangle) \\
&= \langle g(x'), g(x_1), \dots, g(x_q)\rangle - \langle f(x'), g(x_1), \dots, g(x_q)\rangle \\
&\qquad+ \sum_{i=1}^q (-1)^{i+1} \langle f(x'),g(x'),g(x_1),\dots, \widehat {g(x_i)}, \dots, g(x_q) \rangle 
\end{align*}
Since $g(x_i)=f(x_i)$, the above simplifies to:
\[ \pd (P(\sigma)) = g_\#(\sigma) - f_\#(\sigma) + \sum_{i=1}^q (-1)^{i+1} \langle f(x'),g(x'),g(x_1),\dots, \widehat {g(x_i)}, \dots, g(x_q) \rangle. \]

To prove \eqref{prismeq}, it suffices to show that the summation above equals $-P(\pd(\sigma))$. We have:
\begin{align*}
P(\pd(\sigma)) &= P\left( \langle x_1,\dots,x_q\rangle + \sum_{i=1}^q (-1)^i \langle x', x_1,\dots, \widehat{x_i}, \dots, x_q\rangle \right) \\
&= P\left(\sum_{i=1}^q (-1)^i \langle x', x_1,\dots, \widehat{x_i}, \dots, x_q\rangle \right) 
\end{align*}
where $P(\langle x_1,\dots,x_q\rangle) = 0$ since this simplex does not use $x'$. Now when we apply $P$ above, the only nonzero terms are those with $j=0$ in the definition of $P$. All others will repeat some point $f(x_i)=g(x_i)$. Thus we have:
\begin{align*}
P(\pd(\sigma)) &= \sum_{i=1}^q (-1)^i \langle f(x'),g(x'),g(x_1),\dots,\widehat{g(x_i)}, \dots, g(x_n) \rangle \\
&= -  \sum_{i=1}^q (-1)^{i+1} \langle f(x'),g(x'),g(x_1),\dots,\widehat{g(x_i)}, \dots, g(x_n) \rangle 
\end{align*}
and we have proved \eqref{prismeq}. 

The formula \eqref{prismeq} holds when $\sigma$ is any simplex, and so by linearity it will hold for any chain. Now let $\alpha\in Z_q(X)$ be a $q$-cycle, so $\pd(\alpha) = 0$ and thus $P(\pd(\alpha))=0$. Then by \eqref{prismeq} we have:
\[ g_\#(\alpha) - f_\#(\alpha) = \pd P(\alpha) \in B_q(Y). \]
Thus $g_\#(\alpha)$ and $f_\#(\alpha)$ differ by a $q$-boundary, that is, 
$g_*(\alpha) = f_*(\alpha) \in H_q(Y)$.
\end{proof}

Theorem \ref{prism}, and its requirement that the homotopy be strong, suggests that the canonical setting for digital homology should be continuous functions and strong homotopies, rather than the traditional focus on continuous functions and homotopies.

Theorem \ref{prism} allows us to prove that homology groups are preserved by strong homotopy equivalence, defined using strong homotopy in place of ordinary homotopy in the standard definition of homotopy equivalence.
\begin{defn}
Digital images $X$ and $Y$ are \emph{strong homotopy equivalent} when there are continuous functions $f:X\to Y$ and $g:Y \to X$ such that $f \circ g \cong \id_Y$ and $g\circ f \cong \id_X$.
\end{defn}

By standard arguments one can easily show that strong homotopy equivalence is an equivalence relation. 

\begin{cor}
Let $X$ be finite, and let $f:X\to Y$ be a strong homotopy equivalence. Then $f_*:H_q(X) \to H_q(Y)$ is an isomorphism for each $q$.
\end{cor}
\begin{proof}
Let $g:Y\to X$ be continuous with $f\circ g \cong \id_Y$ and $g\circ f \cong \id_X$. By functoriality of the induced homomorphism on homology, we have $(f\circ g)_* = f_* \circ g_*$, and thus by Theorem \ref{prism} we have $f_* \circ g_* = (\id_{Y})_* = \id_{H_q(Y)}$. Similarly $g_* \circ f_* = \id_{H_q(X)}$.
Thus the homomorphisms $f_*:H_q(X)\to H_q(Y)$ and $g_*:H_q(Y) \to H_q(X)$ are inverses, and so both are isomorphisms. 
\end{proof}

The example of $C_4$ again shows the the ``strong'' assumption above is necessary. The constant map $c:C_4 \to \{c\}$ is a homotopy equivalence from $C_4$ to a point, but $c_*: H_1(C_4) \to H_1(\{c\})$ is not an isomorphism since $H_1(C_4) = \Z$ and $H_1(\{c\}) = 0$. 

Since any isomorphism of digital images is automatically a strong homotopy equivalence, we have:
\begin{cor}
Let $f:X\to Y$ be an isomorphism of finite digital images. Then $f_*:H_q(X) \to H_q(Y)$ is an isomorphism for each $q$.
\end{cor}

The converse of Theorem \ref{prism} is not true. That is, it is possible for the homomorphism to be invariant even if the homotopy is not strong. We will show that the induced homomorphisms for selfmaps of $C_n$ with $n>4$ are invariant even for homotopies which are not strong.

In the case of $C_n$, we can make a full computation of the induced homomorphisms for any selfmap, using results from \cite{bs19}. Let $c:C_n \to C_n$ be the constant map $c(c_i) = c_0$, let $l:C_n \to C_n$ be the ``flip map'' $l(c_i) = c_{-i}$, and for some integer $d$, let $r_d:C_n \to C_n$ be the rotation $r_d(c_i) = c_{i+d}$. Theorem 9.3 of \cite{bs19} shows that there are exactly 3 homotopy classes of selfmaps on $C_n$: any map $f:C_n \to C_n$ is either strongly homotopic to a constant, or is homotopic to the identity and equals $r_d$ for some $d$, or is homotopic to the flip map and equals $r_d \circ l$ for some $d$. (The strongness of the homotopy to the constant was not mentioned in \cite{bs19}, but the homotopy demonstrated in the proof in that paper is easily made punctuated and therefore strong.)

\begin{thm}\label{Cninduced}
Let $n>4$, and let $f:C_n \to C_n$ be continuous. Then for all $q>1$, the induced homomorphism $f_{*,q}:H_q(C_n) \to H_q(C_n)$ is trivial, for $q=0$ the induced homomorphism $f_{*,0} = \id$, and for $q=1$ we have:
\[ f_{*,1} = \begin{cases} \id & \text{ if $f \simeq \id_{C_n},$} \\
-\id &\text{ if $f\simeq l$,} \\
0 &\text{ if $f \simeq c$.} \end{cases} \]
\end{thm}
\begin{proof}
For $q>1$ we have already seen that $H_q(C_n)$ is a trivial group, so we will have $f_*=0$ automatically.
When $f$ is homotopic to a constant, as mentioned above, in fact $f \cong c$ and thus $f_*=c_*$ for all $q$, and so $f_{*,0}$ is the identity and $f_{*,q}$ is trivial for $q>0$.

Now we consider when $f$ is homotopic to the identity or the flip map, for $q\in \{0,1\}$. First we consider $q=0$ and $f\simeq \id_{C_n}$. Since $H_0(C_n)$ is generated by $\langle c_0 \rangle$, it suffices to show that $f_*(\langle c_0\rangle) = \langle c_0 \rangle$. Let $d$ be some integer with $f= r_d$, and we have:
\[ f_*(\langle c_0 \rangle) = \langle f(c_0) \rangle = \langle c_d \rangle = \langle c_0 \rangle \]
as desired. Exactly the same argument applies for $f\simeq l$, since we will still have $\langle f(c_0) \rangle = \langle c_d \rangle$ for some $d$.

Now for $q=1$, and $f\simeq \id_{C_n}$ we must show $f_*(\sigma) = \sigma$, where $\sigma = \sum_{i=0}^{n-1} \langle c_i,c_{i+1} \rangle$ as in the proof of Theorem \ref{Cnhomology}. Let $f = r_d$, and we have:
\[ f(\sigma) =  \sum_{i=0}^{n-1} \langle f(c_i), f(c_{i+1}) \rangle = \sum_{i=0}^{n-1} \langle c_{i+d},c_{i+1+d} \rangle = \sigma \]
as desired.

Finally we consider $q=1$ and $f\simeq l$, and we must show $f_*(\sigma) = -\sigma$. Let $f=r_d \circ l$, so $f(c_i) = c_{d-i}$, and we have:
\begin{align*} 
f(\sigma) &= \sum_{i=0}^{n-1} \langle f(c_i), f(c_{i+1}) \rangle =  \sum_{i=0}^{n-1} \langle c_{d-i},c_{d-i-1} \rangle = \sum_{i=0}^{n-1} \langle c_{-i+1},c_{-i} \rangle  \\
&= \sum_{i=0}^{n-1} \langle c_{i+1},c_{i} \rangle = \sum_{i=0}^{n-1} -\langle c_i,c_{i+1}\rangle = -\sigma
\end{align*}
as desired.
\end{proof}

The three cases of Theorem \ref{Cninduced} suffice to compute $f_*$ for any selfmap of $C_n$, and we note that in each of the three homotopy classes, the set of induced homomorphisms is different. We obtain a sort of Hopf theorem for digital cycles:
\begin{cor}
Let $n> 4$, and let $f,g:C_n \to C_n$ be continuous. Then $f_{*,q} = g_{*,q}$ for each $q$ if and only if $f\simeq g$. 
\end{cor}

\section{Catalog of small digital images up to strong homotopy equivalence}\label{catalogsec}
The authors of \cite{hmps15} produced a catalog of all digital images up to homotopy equivalence of 7 points or fewer. The present author in \cite{stae15} extended this listing to 9 points, making heavy use of computer enumerations. 

In this section we present a similar listing of all digital images up to strong homotopy equivalence of 9 points or fewer. Again we rely on computer search, though as we will see the search in the present case is much easier than that required in \cite{hmps15} and \cite{stae15}.

Following a definition in \cite{hmps15}, we define:

\begin{defn}
A finite digital image $X$ is \emph{strongly reducible} when it is strong homotopy equivalent to a digital image of fewer points.
\end{defn}

Theorem 2.10 of \cite{hmps15} showed that if $X$ and $Y$ are not reducible, then $X$ and $Y$ are homotopy equivalent if and only if they are isomorphic. Since any strongly reducible image is reducible, we obtain: 
\begin{lem}
If $X$ and $Y$ are finite and not strongly reducible, then they are strong homotopy equivalent if and only if they are isomorphic.
\end{lem}

Thus if we wish to enumerate all connected digital images up to strong homotopy equivalence, it suffices to begin with an enumeration of all connected graphs up to isomorphism, and discard those graphs which are strongly reducible. We focus now on determining whether or not a given image is strongly reducible.

We begin by stating ``strong'' analogues of several results from \cite{hmps15}. First we have an analogue of \cite[Lemma 2.8]{hmps15}.
\begin{lem}\label{nonsurjlem}
A finite image $X$ is strongly reducible if and only if $\id_X$ is strongly homotopic to a nonsurjective map.
\end{lem}
\begin{proof}
We mimic the proof from \cite{hmps15}. First assume that $\id_X$ is strongly homotopic in one step to a nonsurjection $f:X\to X$, and let $Y=f(X)$ so we may consider $f$ as a map $f:X\to Y$. If $i:Y\to X$ is the inclusion mapping, then $i \circ f = \id_Y$ and $f\circ i = f \cong \id_X$. Thus, $X$ is strong homotopy equivalent to $Y \subsetneq X$, and so $X$ is strongly reducible.

For the converse, assume that $X$ is strongly reducible, so $X$ is strong homotopy equivalent to an image $Z$ with $\#Z < \#X$. Thus there are maps $f:X\to Z$ and $g:Z \to X$ with $g\circ f \cong \id_X$. But we have
\[ \#g(f(X)) \le \#f(X) \le \#Z < \#X, \]
and so $g\circ f$ must be nonsurjective, and thus $\id_X$ is strongly homotopic to a nonsurjection.
\end{proof}

The lemma above can be strengthened: the homotopy can be taken to be punctuated, and in one step. 
\begin{lem}\label{nonsurjlem1steppunc}
A finite image $X$ is strongly reducible if and only if $\id_X$ is homotopic in one step by punctuated homotopy to a nonsurjective map.
\end{lem}
\begin{proof}
We use a variation on the proof of \cite[Lemma 2.9]{hmps15}. If $\id_X$ is homotopic by punctuated homotopy in one step to a nonsurjection, then $X$ is strongly reducible by Lemma \ref{nonsurjlem}. So we need only show that if $X$ is strongly reducible, then $\id_X$ is homotopic by punctuated homtopy in one step to a nonsurjection.

If $X$ is strongly reducible then by Lemma \ref{nonsurjlem} we have some nonsurjection $f:X\to X$ with $\id_X \cong f$. Let $H:X\times [0,k]_\Z \to X$ be the strong homotopy from $\id_X$ to $f$, and by Theorem \ref{punctuatedthm} we may assume that $H$ is punctuated. Without loss of generality assume that $H(x,t)$ is surjective for all $t<k$ and nonsurjective only for $t=k$. Let $g:X\to X$ be the surjective map $g(x) = H(x,k-1)$, so $g$ is homotopic by punctuated homotopy to $f$ in one step. Since $g$ is a continuous surjective selfmap, it is a bijection and thus an isomorphism by \cite[Lemma 2.3]{hmps15}. 

Now consider the homotopy $L:X\times [0,1]_\Z \to X$ where $L(x,0)=\id_X$ and $L(x,1)=f(g^{-1}(x))$. Because $f$ is nonsurjective, the map $f\circ g^{-1}$ is nonsurjective, and thus it remains only to show that $L$ is a punctuated homotopy. That is, we must show that $f\circ g^{-1}$ fixes all but one point of $X$.

Since $g$ is homotopic to $f$ in one step by punctuated homotopy, we have $f(x)=g(x)$ for all $x\in X$ except for one point. Since $g$ is an isomorphism, we can write any point $x \in X$ as $x=g^{-1}(y)$ for some $y$. Thus we have $f(g^{-1}(y)) = g(g^{-1}(y) = y$ for all $y\in X$ except for the point $b = g^{-1}(a)$. Thus $L$ is a punctuated one step homotpy from $\id_X$ to $f\circ g^{-1}$ as desired.
\end{proof}

The catalog in \cite{hmps15} was obtained with a mixture of computer search and by-hand analysis. The following lemma was used as part of the computer search. Given $x\in X$, let $N^*(x)\subseteq X$ be the set of all $y\in X$ with $y\adjeq x$. 
\begin{lem}[\cite{hmps15}, Lemma 4.2]\label{subneighborlem}
If there exists distinct $x,y\in X$ such that $N^*(x)\subseteq N^*(y)$, then $X$ is reducible.
\end{lem}

The converse is not true: for example $C_4$ is reducible but does not satisfy the condition of Lemma \ref{subneighborlem}. The computer search used in \cite{hmps15} and \cite{stae15} thus had to rely in some special cases on a brute-force checking of all maps homotopic to the identity in one step, to determine if any were nonsurjective.

For the case of strong homotopy, however, we can avoid the cumbersome brute-force search because the converse to Lemma \ref{subneighborlem} will hold:
\begin{thm}\label{strongsubnbhd}
Let $X$ be a finite digital image. Then $X$ is strongly reducible if and only if there exists distinct $x,y\in X$ such that $N^*(x)\subseteq N^*(y)$.
\end{thm}
\begin{proof}
First assume that there are distinct $x,y\in X$ with $N^*(x)\subseteq N^*(y)$. Then we follow the proof of \cite[Lemma 4.2]{hmps15} to show that $X$ is strongly reducible. Let $f:X\to X-\{x\}$ be the map which sends $x$ to $y$ and fixes all other points. 

We will show that $f$ is strongly homotopic to $\id_X$ using Theorem \ref{stronghtpadjs}. If we take any $a, b\in X$ with $a \adj b$, we must show that $a\adjeq f(b)$. When $b\neq x$ we have $f(b)=b$ and so $a\adjeq f(b)$ since $a\adj b$. When $b=x$, since $N^*(x)\subseteq N^*(y)$ we have $a \adjeq y = f(x) = f(b)$ as desired. Thus by Theorem \ref{stronghtpadjs} we have $f\cong \id_X$, and since $f$ is not a surjection, $X$ is strongly reducible.

Now we prove the converse. Assume that $X$ is strongly reducible, so by Theorem \ref{nonsurjlem1steppunc} there is a nonsurjection $f$ with $f\cong \id_X$ by a punctuated one step homotopy. Since the homotopy is punctuated and in one step, there is one point $a$ with $f(a)\adj a$ but $f(a)\neq a$, and $f(x)=x$ for all other points. 

Let $b=f(a)$, so $b\adj a$ but $b\neq a$. To show that $N^*(a)\subseteq N^*(b)$, 
take any $x\adjeq a$, and we must show $x\adjeq b$. Since $x\adjeq a$ and $f$ is strongly homotopic to the identity in one step, by Theorem \ref{stronghtpadjs} we have $x \adjeq f(a)=b$ as desired.
\end{proof}

The condition of Theorem \ref{strongsubnbhd} can be easily checked by computer. Starting with a list of all connected simple graphs, we remove those which are strongly reducible. For images of 6 points or fewer, the complete catalog is shown in Figure \ref{catalog}. Source code in SageMath for the computer search is available at the author's website.\footnote{\url{http://faculty.fairfield.edu/cstaecker}}

\newcommand{\ra}{.4}
\newcommand{\rpoly}[2][0]{%
\foreach \x in {1,...,#2} {
\node at ({\x*360/#2 + #1}:\ra) [vertex] {};
\draw ({\x*360/#2 + #1}:\ra) -- ({(\x+1)*360/#2 + #1}:\ra);
}}

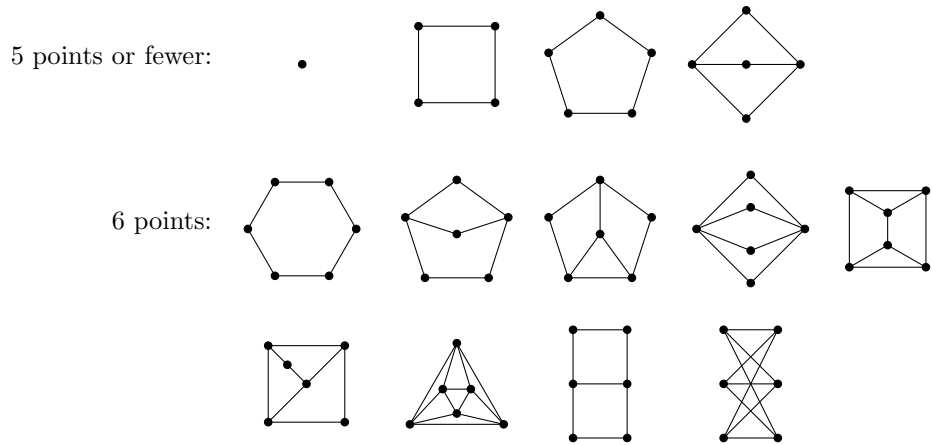
\begin{figure}
\def\arraystretch{1.5}
\makebox[\textwidth][c]{
\begin{tabular}{rcccccc}
5 points or fewer: &
\raisebox{-0.5\height}{\begin{tikzpicture}[scale=1.8]
\node at (0.5,0.5) [vertex] {};
\end{tikzpicture}}
\vspace{.05in} 
&
\raisebox{-0.5\height}{\begin{tikzpicture}[scale=1.8]
\rpoly[45]{4}
\end{tikzpicture}}
&
\raisebox{-0.5\height}{\begin{tikzpicture}[scale=1.8]
\rpoly[18]{5}
\end{tikzpicture}}
& 
\raisebox{-0.5\height}{\begin{tikzpicture}[scale=1.8]
\rpoly{4}
\node at (0,0) [vertex]{};
\draw (-1*\ra,0) -- (1*\ra,0);
\end{tikzpicture}}
\vspace{.2in} \\ 
6 points: &
\raisebox{-0.5\height}{\begin{tikzpicture}[scale=1.8]
\rpoly{6}
\end{tikzpicture}}
\vspace{.2in}
&
\raisebox{-0.5\height}{\begin{tikzpicture}[scale=1.8]
\rpoly[18]{5}
\node at (0,0) [vertex] {};
\draw (162:\ra) -- (0,0) -- (18:\ra);
\end{tikzpicture}}
& 
\raisebox{-0.5\height}{\begin{tikzpicture}[scale=1.8]
\rpoly[18]{5}
\node at (0,0) [vertex] {};
\foreach \t in {90,234,306} {
 \draw (0,0) -- (\t:\ra);
}
\end{tikzpicture}}
& 
\raisebox{-0.5\height}{\begin{tikzpicture}[scale=1.8]
\rpoly{4}
\node at (0,.4*\ra) [vertex] {};
\node at (0,-.4*\ra) [vertex] {};
\draw (180:\ra) -- (0,.4*\ra) -- (0:\ra);
\draw (180:\ra) -- (0,-.4*\ra) -- (0:\ra);
\end{tikzpicture}}
&
\raisebox{-0.5\height}{\begin{tikzpicture}[scale=1.8]
\rpoly[45]{4}
\foreach \x in {-1,1} {
 \node at (0,\x*.3*\ra) [vertex] {};
 \draw (45:\x*\ra) -- (0,\x*.3*\ra) -- (135:\x*\ra);
}
\draw (0,-.3*\ra) -- (0,.3*\ra);
\end{tikzpicture}} \\
& 
\raisebox{-0.5\height}{\begin{tikzpicture}[scale=1.8]
\rpoly[45]{4}
\node at (0,0) [vertex] {};
\node at (135:.5*\ra) [vertex] {};
\draw (-135:\ra) -- (45:\ra);
\draw (135:\ra) -- (0,0);
\end{tikzpicture}}
& 
\raisebox{-0.5\height}{\begin{tikzpicture}[scale=1.8]
\rpoly[90]{3}
\foreach \t in {0,120,240} {
 \node at (\t+30:.3*\ra) [vertex] {};
 \draw (\t+90:\ra) -- (\t+30:.3*\ra) -- (\t-30:\ra);
 \draw (\t+30:.3*\ra) -- (\t+150:.3*\ra);
}
\end{tikzpicture}}
&   
\raisebox{-0.5\height}{\begin{tikzpicture}[scale=1.8]
\draw[step=\ra] (0,0) grid (1*\ra,2*\ra);
\foreach \x in {0,1} {
 \foreach \y in {0,1,2} {
  \node at (\x*\ra,\y*\ra) [vertex] {};
  }
}
\end{tikzpicture}}
&
\raisebox{-0.5\height}{\begin{tikzpicture}[scale=1.8]
\foreach \x in {0,1} {
 \foreach \y in {0,1,2} {
  \node at (\x*\ra,\y*\ra) [vertex] {};
 }
}
\foreach \y in {0,1,2} {
 \foreach \z in {0,1,2} {
  \draw (0,\y*\ra) -- (1*\ra,\z*\ra);
 }
}
\end{tikzpicture}}
\\ 
\end{tabular}}
\caption{All digital images on 6 points or fewer, up to strong homotopy equivalence\label{catalog}}
\end{figure}

For images of 7, 8, and 9 points, the results of the computer search are as follows. Let $c(n)$ be the number of connected digital images on $n$ points which are not reducible. Let $d(n)$ be the number of connected digital images which are not strongly reducible. We have:
\[ 
\begin{tabular}{r|ccccccccc}
$n$ & 1 & 2 & 3 & 4 & 5 & 6 & 7 & 8 & 9 \\
\hline
$c(n)$ & 1 & 0 & 0 & 0 & 1 & 1 & 3 & 28 & 547 \\
$d(n)$ & 1 & 0 & 0 & 1 & 2 & 9 & 46 & 507 & 11800
\end{tabular}
\]
The first row is from \cite{stae15}, and is OEIS sequence A248571.

\section{Pointed strong homotopy equivalence}\label{pointedsec}
The study of pointed digital images and pointed homotopy arises naturally in the study of the fundamental group of a digital image, see \cite{boxe05}. 
A \emph{pointed digital image} $(X,x_0)$ is a digital image with some specific chosen point $x_0\in X$. A \emph{pointed continuous function} $f:(X,x_0) \to (Y,y_0)$ is a function which is continuous and $f(x_0) = y_0$. Two pointed continuous functions are \emph{pointed homotopic} when there is a homotopy from $f$ to $g$ through pointed functions. Two pointed images $(X,x_0)$ and $(Y,y_0)$ are \emph{pointed homotopy equivalent} when there are pointed continuous maps $f:(X,x_0) \to (Y,y_0)$ and $g:(Y,y_0) \to (X,x_0)$ such that $g\circ f$ is pointed homotopic to $\id_X$ and $f\circ g$ is pointed homotopic to $\id_Y$. We will say two pointed images $(X,x_0)$ and $(Y,y_0)$ are \emph{pointed strong homotopy equivalent} when there are pointed continuous maps $f:(X,x_0) \to (Y,y_0)$ and $g:(Y,y_0) \to (X,x_0)$ such that $g\circ f$ is pointed homotopic by a strong homotopy to $\id_X$ and $f\circ g$ is pointed homotopic by a strong homotopy to $\id_Y$.

In \cite{hmps15} an example was given of two digital images which are homotopy equivalent but not pointed homotopy equivalent for any choice of points:
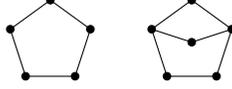
\[ 
\begin{tikzpicture}[scale=1.4]
\rpoly[18]{5} 
\end{tikzpicture} 
\qquad
\begin{tikzpicture}[scale=1.4]
\rpoly[18]{5} 
\node at (0,0) [vertex] {};
\draw (162:\ra) -- (0,0) -- (18:\ra);
\end{tikzpicture} 
\]
In fact these two images are not strong homotopy equivalent-- we can see both appear in Figure \ref{catalog}. Thus they also are not pointed strong homotopy equivalent.

It is natural to ask whether there are examples of pointed digital images $(X,x_0)$ and $(Y,y_0)$ which are strong homotopy equivalent but not pointed strong homotopy equivalent. We will give such an example in Example \ref{pointedexl}.

First we prove a theorem which will help with the example. We say a digital image is \emph{strongly contractible} when it is strong homotopy equivalent to a point.

Given a finite image $X$ with $\#X=n$, let $(x_1,\dots, x_n)$ be an ordered enumeration of the points of $X$. For $i \in \{1,\dots,n\}$, let $X_i \subseteq X$ be the set $X_i = \{x_{i},\dots, x_n\}$. For $x\in X_i$, let $N^*_{X_i}(x)$ be the set of all $y\in X_i$ with $y\adjeq x$. 

We say $(x_1,\dots, x_n)$ is a \emph{strong contraction ordering of $X$} when for each $i<n$, there is some $j>i$ with $N_{X_i}^*(x_i) \subseteq N_{X_i}^*(x_j)$. Such an ordering of $X$ may or may not exist. 

\begin{thm}\label{orderthm}
A finite digital image $X$ is strongly contractible if and only if there is a strong contraction ordering of $X$.
\end{thm}
\begin{proof}
First we assume $X$ is strongly contractible. We will prove that there is a strong contraction ordering by induction on $n=\#X$.  If $n = 1$ then clearly there is a strong contraction ordering. For the inductive step, if $X$ is strongly contractible then it is strongly reducible, and thus by Theorem \ref{strongsubnbhd} there is some $x_1\in X$ and $y\in X$ with $N^*(x_1) \subseteq N^*(y)$, and by the proof of Theorem \ref{strongsubnbhd} $X$ strongly reduces to $X - \{x_1\}$. By induction, $X-\{x_1\}$ has a strong contraction ordering $(x_2,\dots,x_n)$, and then $(x_1,x_2,\dots,x_n)$ is a strong contraction ordering of $X$.

Now we assume that $X$ has a strong contraction ordering $(x_1,\dots,x_n)$. We use induction on $n$ to prove that $X$ is strongly contractible. When $n=1$ then clearly $X$ is contractible. For the inductive case note that $(x_2,\dots,x_n)$ is a strong contraction ordering of $X-\{x_1\}$, and so by induction $X-\{x_1\}$ is strongly contractible. Since $(x_1,\dots,x_n)$ is a strong contraction ordering there is some $j>1$ with $N^*(x_1) \subseteq N^*(x_j)$, and thus by Theorem \ref{strongsubnbhd} and its proof, $X$ is strong homotopy equivalent to $X-\{x_1\}$. Since $X-\{x_1\}$ is strongly contractible, also $X$ is strongly contractible as desired.
\end{proof}

\begin{figure}
\[
\begin{tikzpicture}
\node at (3,0) [vertex, label=right:{$x_1$}] {};
\node at (3,3) [vertex, label=right:{$x_2$}] {};
\node at (0,3) [vertex, label=left:{$x_3$}] {};
\node at (0,0) [vertex, label=left:{$x_4$}] {};
\node at (2,2) [vertex, label=right:{$x_5$}] {};
\node at (1,2) [vertex, label=left:{$x_6$}] {};
\node at (1,1) [vertex, label=left:{$x_7$}] {};
\node at (2,1) [vertex, label=right:{$x_8$}] {};
\draw (0,0) -- (3,0) -- (3,3) -- (0,3) -- (0,0);
\draw (1,1) grid (2,2);
\draw (0,0) -- (3,3);
\draw (0,3) -- (1,2);
\draw (2,1) -- (3,0);
\draw (0,0) -- (1,2);
\draw (0,3) -- (2,2);
\draw (3,3) -- (2,1);
\draw (0,0) -- (2,1);
\draw (1,1) -- (3,0);
\end{tikzpicture}
\]
\caption{A pointed digital image which is strong homotopy equivalent to a point, but not pointed strong homotopy equivalent to a point.\label{pointedfig}}
\end{figure}
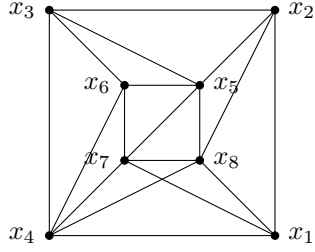

Now we present our example of a pointed image which is strongly contractible but not pointed strongly contractible. We do not know if there is a simpler example, or one which can be embedded into $\Z^2$ or $\Z^3$ with some $c_u$-adjacency.
 
\begin{exl}\label{pointedexl}
Let $(X,x_1)$ be the pointed digital image shown in Figure \ref{pointedfig}. This digital image is strong homotopy equivalent to a point, but not by a pointed strong homotopy. To show that $X$ is strong homotopy equivalent to a point, by Theorem \ref{orderthm} it suffices to demonstrate a strong contraction ordering of $X$. Using the labeling of points in Figure \ref{pointedfig}, we claim that $(x_1,x_2,\dots,x_8)$ is a strong contraction ordering of $X$.

The various $X_i$ are shown in Figure \ref{Xifig}. 
It is routine to verify the following inclusions:
\begin{align*}
N^*_{X_1}(x_1) \subseteq N^*_{X_1}(x_8), \, N^*_{X_2}(x_2) \subseteq N^*_{X_2}(x_5), \, N^*_{X_3}(x_3) \subseteq N^*_{X_3}(x_6), \\
N^*_{X_4}(x_4) \subseteq N^*_{X_4}(x_7), \, N^*_{X_5}(x_5) \subseteq N^*_{X_5}(x_7), \, N^*_{X_6}(x_6) \subseteq N^*_{X_6}(x_7), \\
N^*_{X_7}(x_7) \subseteq N^*_{X_7}(x_8),
\end{align*}
and thus $(x_1,x_2,\dots,x_8)$ is a strong contraction ordering, and so $X$ is strongly contractible.
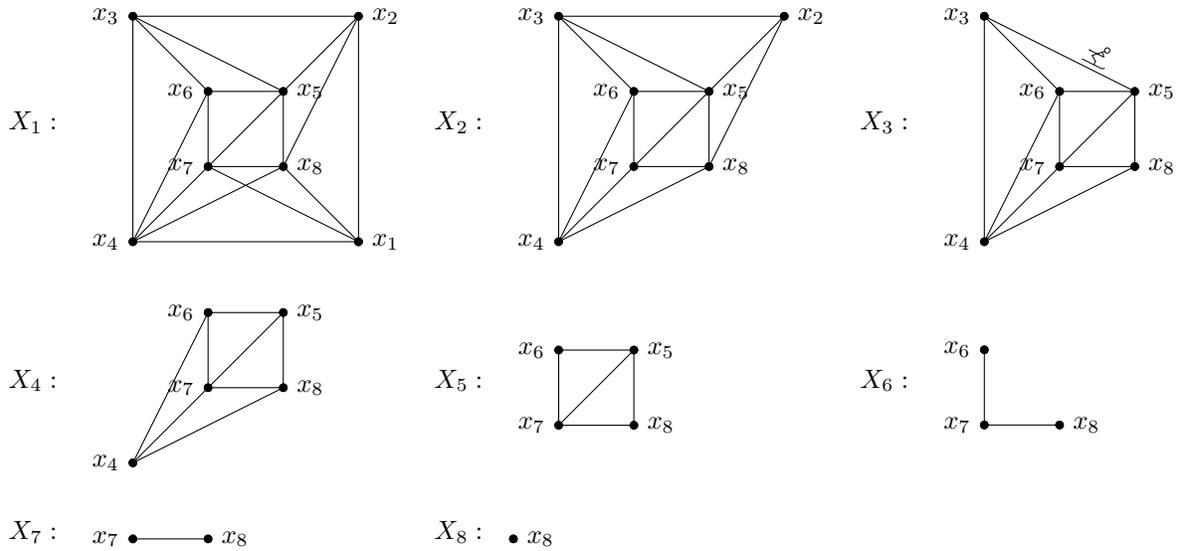
\begin{figure}
\makebox[\textwidth][c]{
$\begin{array}{rlrlrl}
X_1:& \raisebox{-0.5\height}{\begin{tikzpicture}
\node at (3,0) [vertex, label=right:{$x_1$}] {};
\node at (3,3) [vertex, label=right:{$x_2$}] {};
\node at (0,3) [vertex, label=left:{$x_3$}] {};
\node at (0,0) [vertex, label=left:{$x_4$}] {};
\node at (2,2) [vertex, label=right:{$x_5$}] {};
\node at (1,2) [vertex, label=left:{$x_6$}] {};
\node at (1,1) [vertex, label=left:{$x_7$}] {};
\node at (2,1) [vertex, label=right:{$x_8$}] {};
\foreach \x in {0,3} {
 \foreach \y in {0,3} {
  \node at (\x,\y) [vertex] {};
 }
}
\foreach \x in {1,2} {
 \foreach \y in {1,2} {
  \node at (\x,\y) [vertex] {};
 }
}
\node at (3,0) [vertex] {};
\draw (0,0) -- (3,0) -- (3,3) -- (0,3) -- (0,0);
\draw (1,1) grid (2,2);
\draw (0,0) -- (3,3);
\draw (0,3) -- (1,2);
\draw (2,1) -- (3,0);
\draw (0,0) -- (1,2);
\draw (0,3) -- (2,2);
\draw (3,3) -- (2,1);
\draw (0,0) -- (2,1);
\draw (1,1) -- (3,0);
\end{tikzpicture}}
& 
X_2:& \raisebox{-0.5\height}{\begin{tikzpicture}
\node at (3,3) [vertex, label=right:{$x_2$}] {};
\node at (0,3) [vertex, label=left:{$x_3$}] {};
\node at (0,0) [vertex, label=left:{$x_4$}] {};
\node at (2,2) [vertex, label=right:{$x_5$}] {};
\node at (1,2) [vertex, label=left:{$x_6$}] {};
\node at (1,1) [vertex, label=left:{$x_7$}] {};
\node at (2,1) [vertex, label=right:{$x_8$}] {};
\node at (0,0) [vertex] {};
\node at (0,3) [vertex] {};
\node at (3,3) [vertex] {};
\foreach \x in {1,2} {
 \foreach \y in {1,2} {
  \node at (\x,\y) [vertex] {};
 }
}
\draw (0,0) -- (0,3) -- (3,3);
\draw (0,3) -- (1,2);
\draw (1,1) grid (2,2);
\draw (0,0) -- (3,3);
\draw (0,0) -- (1,2);
\draw (0,3) -- (2,2);
\draw (3,3) -- (2,1);
\draw (0,0) -- (2,1);
\end{tikzpicture}}
&  
X_3:& \raisebox{-0.5\height}{\begin{tikzpicture}
\node at (0,3) [vertex, label=left:{$x_3$}] {};
\node at (0,0) [vertex, label=left:{$x_4$}] {};
\node at (2,2) [vertex, label=right:{$x_5$}] {};
\node at (1,2) [vertex, label=left:{$x_6$}] {};
\node at (1,1) [vertex, label=left:{$x_7$}] {};
\node at (2,1) [vertex, label=right:{$x_8$}] {};
\node at (0,0) [vertex] {};
\node at (0,3) [vertex] {};
\foreach \x in {1,2} {
 \foreach \y in {1,2} {
  \node at (\x,\y) [vertex] {};
 }
}
\draw (0,0) -- (0,3);
\draw (0,3) -- (1,2);
\draw (1,1) grid (2,2);
\draw (0,0) -- (2,2);
\draw (0,0) -- (1,2);
\draw (0,3) -- (2,2);
\draw (0,0) -- (2,1);
\begin{scope}[transform canvas={shift={(1.5,2.5)},scale=.05,rotate=-26.6}]
\coordinate (a) at (-3,-3);
\coordinate (b) at (.5,-3);
\coordinate (c) at (3,-3);
\coordinate (d) at (3.5,-2.5);
\coordinate (e) at (.5,-2);
\coordinate (f) at (-1,-1);
\coordinate (h) at (2,0);
\coordinate (i) at (1,1);
\coordinate (j) at (2,2);
\coordinate (k) at (-2.5,0);
\coordinate (l) at (-2,.2);
\coordinate (m) at (-2,-.2);
\draw (a) -- (c) -- (d);
\draw (b) -- (e) -- (f) -- (j);
\draw (h) -- (k);
\draw (h) -- (i);
\draw (l) -- (m);
\filldraw[fill=white, draw=black] (j) circle (.7);
\end{scope}

\end{tikzpicture}}
\vspace{.5cm}
\\
X_4:&
\raisebox{-0.5\height}{\begin{tikzpicture}
\node at (0,0) [vertex, label=left:{$x_4$}] {};
\node at (2,2) [vertex, label=right:{$x_5$}] {};
\node at (1,2) [vertex, label=left:{$x_6$}] {};
\node at (1,1) [vertex, label=left:{$x_7$}] {};
\node at (2,1) [vertex, label=right:{$x_8$}] {};
\node at (0,0) [vertex] {};
\foreach \x in {1,2} {
 \foreach \y in {1,2} {
  \node at (\x,\y) [vertex] {};
 }
}
\draw (1,1) grid (2,2);
\draw (0,0) -- (2,2);
\draw (0,0) -- (1,2);
\draw (0,0) -- (2,1);
\end{tikzpicture}}
&  X_5:& 
\raisebox{-0.5\height}{\begin{tikzpicture}
\node at (2,2) [vertex, label=right:{$x_5$}] {};
\node at (1,2) [vertex, label=left:{$x_6$}] {};
\node at (1,1) [vertex, label=left:{$x_7$}] {};
\node at (2,1) [vertex, label=right:{$x_8$}] {};
\foreach \x in {1,2} {
 \foreach \y in {1,2} {
  \node at (\x,\y) [vertex] {};
 }
}
\draw (1,1) grid (2,2);
\draw (1,1) -- (2,2);
\end{tikzpicture}}
& X_6:&
\raisebox{-0.5\height}{\begin{tikzpicture}
\node at (0,1) [vertex, label=left:{$x_6$}] {};
\node at (0,0) [vertex, label=left:{$x_7$}] {};
\node at (1,0) [vertex, label=right:{$x_8$}] {};
\node at (0,0) [vertex] {};
\node at (1,0) [vertex] {};
\node at (0,1) [vertex] {};
\draw (0,0) -- (0,1);
\draw (1,0) -- (0,0);
\end{tikzpicture}}
\vspace{.5cm}
\\ X_7:&
\raisebox{-0.5\height}{\begin{tikzpicture}
\node at (0,0) [vertex, label=left:{$x_7$}] {};
\node at (1,0) [vertex, label=right:{$x_8$}] {};
\node at (0,0) [vertex] {};
\node at (1,0) [vertex] {};
\draw (1,0) -- (0,0);
\end{tikzpicture}}
& X_8:&
\raisebox{-0.5\height}{\begin{tikzpicture}
\node at (0,0) [vertex, label=right:{$x_8$}] {};
\node at (0,0) [vertex] {};
\end{tikzpicture}}
\end{array}
$}
\caption{Various $X_i$ for Example \ref{pointedexl}.\label{Xifig}}
\end{figure}

To show that $(X,x_1)$ is not pointed strong homotopy equivalent to a point, it suffices to show that there is no map strongly homotopic in one step to $\id_X$ other than $\id_X$ itself. For the sake of contradiction, assume there is some continuous $f$ strongly homotopic in one step to $\id_X$. By inspecting the adjacencies in Figure \ref{pointedfig} we see that $x_1$ is the only point in $X$ for which we can have $f(x) \adj x$ while still preserving all necessary adjacencies (we must have $y\adjeq f(x)$ for every $y\adj x$). Thus any map strongly homotopic in one step to $\id_X$ must have $f(x_1)\neq x_1$.

Since any map strongly homotopic to $\id_X$ must move $x_1$, any strong homotopy from $\id_X$ to a constant must move $x_1$ at one of its stages, and so it cannot be pointed. Thus we have shown that $(X,x_1)$ is not pointed strong homotopy equivalent to a point. 
\end{exl}

Thus, in general, if $X$ and $Y$ are strong homotopy equivalent and $x_0\in X$ and $y_0\in Y$, it does not automatically follow that $(X,x_0)$ is pointed strong homotopy equivalent to $(Y,y_0)$. The example of \cite{hmps15} gave two images $X$ and $Y$ which are homotopy equivalent but not pointed equivalent for \emph{any} choice of points $x_0 \in X$ and $y_0\in Y$. We do not know if the same type of example is possible for strong homotopy.

\begin{quest}
Are there strong homotopy equivalent digital images $X$ and $Y$ for which $(X,x_0)$ and $(Y,y_0)$ are not strong pointed homotopy equivalent for all $x_0\in X$ and $y_0\in Y$?
\end{quest}

\bibliographystyle{hplain}

\end{document}